\documentclass[10pt]{article}
\usepackage[utf8]{inputenc}
 
\usepackage{amssymb}
\usepackage{amsmath}
\usepackage{amsthm}
\usepackage{url}
\usepackage{mathrsfs}
\usepackage{pdfpages}

\newtheorem{lem}{Lemma}

\newtheorem{theor}{Theorem}
\newtheorem{cor}{Corolary}
\theoremstyle{remark}
\newtheorem{rem}{Remark}

\newcommand{\R}{\mathbb{R}}

\newcommand{\sA}{\mathscr{A}}
\newcommand{\sB}{\mathscr{B}}
\newcommand{\sC}{\mathscr{C}}

\newcommand{\eltwo}{\mathcal{L}_2}

\newcommand{\expect}{\mathbb{E}}

\newcommand{\WGN}{\widehat{G}_N}
\newcommand{\cover}{\widehat{G}}
\newcommand{\entr}{\mathcal{H}}

\DeclareMathOperator{\pr}{pr}

\DeclareMathOperator{\Aut}{Aut}
\DeclareMathOperator{\Ball}{Ball}
\DeclareMathOperator{\Prob}{Pr}
\DeclareMathOperator{\Cov}{Cov}
\DeclareMathOperator{\Var}{Var}

\begin{document}

\author{Andrei Alpeev  \footnote{Chebyshev Laboratory, St. Petersburg State University, 14th Line, 29b, Saint Petersburg, 199178 Russia, alpeevandrey@gmail.com}}
\title{Entropy inequalities and exponential decay of correlations for unique Gibbs measures on trees. \footnote{This research is supported by ``Native towns'', a social investment program of PJSC ``Gazprom Neft''}}

\maketitle

\begin{abstract}
In a recent paper by A. Backhausz, B. Gerencs\'er and V. Harangi, it was shown that factors of independent identically distributed random processes (IID) on trees obey certain geometry-driven inequalities. In particular, the mutual information shared between two vertices decays exponentially, and there is an explicit bound for this decay. In this note we show that all of these inequalities could be verbatim translated to the setting of factors of processes driven by unique Gibbs measures. As a consequence, we show that correlations decay exponentially for unique Gibbs measures on trees.
\end{abstract}

keywords: Gibbs measure, regular tree, exponential decay of correlations, Bethe lattice, random coverings.
\section{Introduction}

The paper \cite{BGH} provided a unified approach to entropy inequalities for factors of IID's on trees. It was shown that all known entropy inequalities in this setting follow by certain combinatorial constructions from the ``general edge-vertex`` entropy inequality. The purpose of this note is to show that this general inequality holds in a broader setting of factors of Gibbs processess whose distribution is a unique Gibbs measure, thus transferring all the entropy inequalities.


Let $T_d$ be a $d$-regular tree with vertex set $V(T_d)$. Let $A$ be a finite set ({\em alphabet}). 
Assume that to each vertex $v$ of the tree a function ({\em potential}) $\psi_v: A^{V(T)} \to \R$ is assigned. We would like to require that this potential is ``local'' and ``symmetric''. 
Let's state these requirements formally. Potential $(\psi_v)_{v \in v(T_d)}$ is local if for each $v \in V(T_d)$ there is a postive integer $D$ and function $\psi'_v: A^{\Ball(v,D)} \to \R$ such that $\psi_v = \psi'_v \circ \pr_D$, here $\pr_{\Ball(v,D)}$ stands for the natural projection $A^{V(T_d)} \to A^{\Ball(v,D)}$.
Note that the action of the automorphism group $\Aut(T_d)$ on $T_d$ is lifted naturally to the action on $A^{V(T_d)}$:
\begin{equation*}
(\gamma\omega)(v) = \omega(\gamma^{-1} v),
\end{equation*}
for each $v \in V(T_d)$ and $\gamma \in \Aut(T_d)$.
By saying that the potential is symmetric we mean that for each $v \in V(T_d)$ and $\gamma \in \Aut(T_d)$ holds $\psi_{\gamma v}(\gamma \omega) = \psi_v(\omega)$. We note that due to the latter requirement, the locality of the potential is uniform: there is a positive integer $R$ such that for each $v \in V(T_d)$ there is a function $\psi^R_v: A^{\Ball(v,R)} \to \R$ satisfying $\psi_v = \psi^R_v \circ \pr_{\Ball(v, R)}$. This means that we can define $\psi_v$ on $A^W$ for any $W \supset \Ball(v,R)$. We denote $\partial^R(W)$ the $R$-neighborhood of $W$.  
 
A Borel probability measure $\mu$ on $A^{V(T_d)}$ is called a {\em Gibbs measure} if for every $\Lambda \Subset V(T_d)$, for any $s \in A^\Lambda$ and $t \in A^{\partial^{2R} \Lambda}$ we have the following for the conditional probability 
\begin{equation*}
\Prob(\pr_{\Lambda}(\omega) = s \vert \pr_{\partial^{2R}\Lambda}(\omega) = t ) = 
\frac{\exp\left( -\sum_{v \in \Lambda \cap \partial^R \Lambda} \psi_v(s \sqcup t)   \right)}
{Z_{\Lambda,t} },
\end{equation*}
where $\omega$ is distributed according to $\mu$, and $Z_{\Lambda,t}$ stands for the normalizing coefficient
\begin{equation*}
\sum_{s' \in A^{\Lambda}}\exp\left( -\sum_{v \in \Lambda \cap \partial^R \Lambda} \psi_v(s' \sqcup t)   \right).
\end{equation*}
A Gibbs process is simply a process whose distribution is a Gibbs measure. 

As we will show in this note, all the inequalities from \cite{BGH} could be transferred to the setting of factors of unique Gibbs measure processes.
We will not re-write all of these, but the most striking one: 

\begin{theor}\label{th: decay of mutual info}
Let $T_d$ be a $d$-regular tree, $d>2$. Let $A$ be a finite set. Let $(\psi)_v$ be an $\Aut(T_d)$-invariant potential such  that there is a unique Gibbs measure for it. Let $(X_v)$ be the process distributed according to this measure. Then the following holds for any two vertices $u$ and $v$ of $T_d$ with distance $l$ between them:
\begin{equation*}
\frac{I(X_u,X_v)}{H(X_v)} \leq \left\lbrace 
\begin{array}{cl}
\frac{2}{d(d-1)^l} & \textrm{if k = 2l + 1 is odd,}\\
\frac{1}{(d-1)^l} & \textrm{if k = 2l is even.}
\end{array}\right.
\end{equation*}

\end{theor}

The interesting thing is that we do not impose any restrictions on the model apart from it being symmetric with respect to the tree automorphisms and having unique Gibbs measure. 

In \cite{BGH} the idea was proposed to show that certain processes (including Gibbs proccesses) are not factors of IID based on their failure to satisfy the entropy inequalities. Our results show that this approach will fail to distinguish factors of IID processes from factors of unique Gibbs measure processes.

Models of statistical physics on regualar trees (Bethe lattices) have been extensively studied, see e.g. \cite{Ba}. 
Gibbs processes on infinite graphs and their approximations by big finite graphs have been studied in \cite{DM10a}, \cite{DM10b}. 
In my paper \cite{A15}, it was shown that for unique Gibbs measures, finite models of infinite Gibbs structures admitting unique Gibbs measure provide a way to compute the sofic entropy.
In works \cite{AuP17} and \cite{A17} the random ordering technique was used to assess the asymptotic behaviour of certain classes of Gibbs processes on sequences of finite graphs. 
 
One particular property of interest is the exponential decay of correlations. By a symmetric observable we mean a collection of functions $(\alpha_v)_{v \in V(T)}$ satisfying the same requirements as potential (it should be symmetric and local), but taking values in a finite set $B$ (we will call it a real observable if it takes values in $\R$). In fact, previous theorem holds for any symmetric observable. From that we may deduce that there is an exponential correlation decay:

\begin{cor}
In the setting of the previous theorem, for any symmetric real observable $(\alpha_u)$ there is a constant $C$ such that 
\begin{equation*}
\Cov_{\nu}(\alpha_u,\alpha_v) \leq C (d-1)^{-k/4}.
\end{equation*}
\end{cor}
\begin{proof}
Let $B$ be the set of values of $\alpha_v$ for some (=any) $v \in V(T)$. Set $B$ is finite. There are two natuaral projections $\pr_1,\pr_2 : B \times B \to B$.
Denote $\eta$ the distribution of $\alpha_u()$ for any $u \in V(T)$. 
For each $u,v \in V(T)$ denote $\xi_{u,v}$ the distribution of $(\alpha_u, \alpha_v)$. Let $M'$ be the set of probability measures $\theta$ on $B \times B$ such that $\pr_1(\theta) = \pr_2(\theta) = \eta$. Note that $M'$ is a finite-dimensional convex set. We note that $\entr$ is a smooth function in a neighborhood of $\eta \otimes \eta$. We claim that there is a constant $C'$ such that 
\begin{equation*}
I(\alpha_u,\alpha_v) \geq C'\lVert \xi_{u,v} - \eta \otimes \eta\rVert^2, 
\end{equation*}
where $\lVert \; \rVert$ stands for the total variation distance. 
This follows from the fact that $\eta \otimes \eta$ is the unique maximal point of function $\xi \mapsto 2 \entr(\eta) - \entr(\xi)$ defined on $M'$, and that the second derrivative of this function is positive in any direction. Next we note that $\Cov(\alpha_u,\alpha_v)$ is an affine function of the distribution of $(\alpha_u,\alpha_v)$, so we get that there is a constant $C$ such that 
\begin{equation*}
\Cov(\alpha_u,\alpha_v) \leq C\big(I(\alpha_u,\alpha_v)\big)^{1/2}.
\end{equation*}
This implies the desired.
\end{proof}

In the paper \cite{BSV15}, a stronger decorrelation bound is obtained for factors of IID process. Various results of this flavour were obatined in \cite{BS14}, \cite{BGHV}, \cite{CHV}, \cite{GH17}. It seems that some of these could be transferred to the setting of factors of Gibbs processes. 

Our proof mostly follows the steps of \cite{BGH}.
We refer the reader to the forementioned paper for the details on how to derive Theorem \ref{th: decay of mutual info} and other entropy inequalities from the ``general edge-vertex inequality'', Theorem \ref{th: general edge-vertex}, presented in Section \ref{sec: edge-vertex inequality}. 
So it is enough to prove this inequality. Again, the basic argument structure is the same. The main difference lies in the proof of our Lemma \ref{lem: nice covering has nice coloring}, which corresponds to Lemma 5.4 of \cite{BGH}. In the factor of IID case of \cite{BGH}, a standard measure concentration inequality yuilded the desired right away. In our case, lack of independence in the base process forced us to opt for a second moment argument coupled with some analysis of unique Gibbs measures. Note that the argument takes a kind of a bootstrap form: a weak decorrelation property is used to prove a quantitative one. 

\section{Preliminaries}
For a metrizable compact space $X$, we denote $M(X)$ the space of Borel probability measures endowed with the weak* topology. For two random variables $X$ and $Y$, their covariance is denoted $\Cov(X,Y)$:
\begin{equation*}
\Cov(X,Y) = \expect XY - \expect X \,\expect Y = \expect\big((X - \expect X)(Y-\expect Y)\big).
\end{equation*}

The sign ``$\Subset$'' stands for ``a finite subset''.
For a graph $G$ we will denote $V(G)$ its vertex set and $E(G)$ its edge set. Denote $Aut(G)$ the automorphism group of graph $G$.  For any vertex $v$ and a positive integer $R$, denote $\Ball(v,R)$ the ball of radius $R$ centered at $v$.


Let $H$ be a locally-finite simple graph.
Let us fix a finite set $A$ (an {\em alphabet}).

Let $W$ be any subset of $V(H)$. Denote $\pr_W: A^{V(H)} \to A^W$ the natural projection map. Denote $\sB(W)$ the $\sigma$-subalgebra on $A ^{V(H)}$ that is the preimage of the Borel $\sigma$-algebra on $A^W$ under the projection map $\pr_W$

Fix a positive integer $D$ (a {\em memory parameter}). A {\em potential} is a collection $(\psi)_{v \in V(H)}$ of functions $A^{\Ball(v,D)} \to \R$. We will also say that $\psi$ is a {\em $D$-potential}, reflecting that its memory parameter is $D$. 
Abusing notation a little, we may assume that $\psi_v$ is defined on $A^W$ for any subset $W$ of $V(H)$  containing $\Ball(v,D)$ (including the whole space $A^{V(H)}$). For a positive integer $R$ and a subset $W$ of $V(H)$, we denote $\partial^R W$ the $R$-neighborhood of set $W$, that is the set of all the vertices outside of $W$ at a distance not bigger than $R$. The potential defines the set of {\em Gibbs measures}. If graph $H$ is finite, then the Gibbs measure $\nu$ is given by the formula:
\begin{equation*}
\nu(\lbrace \omega \rbrace) = \frac{\exp \big(- \sum_{v \in V(H)} \psi_v(\omega)\big)}{Z}, 
\end{equation*}
where $Z$ stands for the normalizing coefficient
\begin{equation*}
\sum_{\omega \in A^{V(H)}}\exp \Big(- \sum_{v \in V(H)} \psi_v(\omega)\Big).
\end{equation*}
For $s \in A^{\Lambda_1}$ and $t \in A^{\Lambda_2}$ with $\Lambda_1 \cap \Lambda_2  = \varnothing$, denote $s \sqcup t \in A^{\Lambda_1 \cup \Lambda_2 }$ their glueing.
If $H$ is infinite, then there could be multiple Gibbs measures. A Borel probability measure $\nu$ on $A^{V(H)}$ is a Gibbs measure for potential $(\psi)_{v\in V(H)}$ if for every $\Lambda \Subset V(H)$, for any $s \in A^\Lambda$ and $t \in A^{\partial^{2D} \Lambda}$, we have the following for the conditional probability: 
\begin{equation*}
\Prob(\pr_{\Lambda}(\omega) = s \vert \pr_{\partial^{2D}\Lambda}(\omega) = t ) = 
\frac{\exp\left( -\sum_{v \in \Lambda \cap \partial^D \Lambda} \psi_v(s \sqcup t)   \right)}
{Z_{\Lambda,t} },
\end{equation*}
where $\omega \in A^{V(H)} $ is distributed according to $\nu$, and $Z_{\Lambda,t}$ stands for the normalizing coefficient
\begin{equation*}
\sum_{s' \in A^{\Lambda}}\exp\left( -\sum_{v \in \Lambda \cap \partial^R \Lambda} \psi_v(s' \sqcup t)   \right).
\end{equation*}
We will say that pair $(H,\psi)$ forms a {\em Gibbs structure}.

\section{Coverings and the edge-vertex inequality for Gibbs measures}\label{sec: edge-vertex inequality}

Let $G$ be a finite simple connected graph. The universal covering $\cover$ of graph $G$ is the tree obtained in the following way. Let $v$ be any vertex from $G$. The vertex set of $\cover$ is the set of all non-backtracking paths on $G$ starting from $v$ (the trivial one included). Two paths are connected by an edge in $\cover$ iff one of these paths is another extended by one edge. We define the natural projection map $\varphi: V(\cover) \to V(G)$ which sends a vertex from $\cover$ to the end of the corresponding non-backtracking path.  Of course, this map could be extended to the set of edges $E(\cover)$ as well. Note that the construction described is unique up to an automorphism.

Consider the subroup $\Gamma_{\varphi}$ of $\Aut(\cover)$ consisting of all $\varphi$-preserving automorphisms of graph $\cover$:
\begin{equation*}
\Gamma_{\varphi} = \big\{ \gamma \in \Aut(\cover) \vert\; \varphi(\gamma v) = \varphi(v) \text{ for any } v \in V(\cover)  \big\}.
\end{equation*}
We will usually write $\Gamma$ instead of $\Gamma_{\varphi}$. 

Let $A$ be a finite alphabet. Note that $\Gamma$ acts on $A^{\cover}$:
\begin{equation*}
(\gamma \omega)(v) = \omega (\gamma^{-1} v),
\end{equation*}
for any $\gamma \in \Gamma$, $v \in V(\cover)$ and $\omega \in A^{V(\cover)}$.

Let $\psi = (\psi)_{v \in V(\cover)}$ be a $\Gamma$-invariant potential on $\cover$. That is, for any $v \in V(\cover)$, any $\gamma \in \Gamma$ and $\omega \in A^{V(\cover)}$ holds:

\begin{equation*}
\psi_{\gamma v} (\gamma \omega) = \psi_v (\omega).
\end{equation*}
If $\nu$ is the unique measure for $(\cover, \psi)$, then $\nu$ is automatically $\Gamma$-invariant; this follows from the fact that $\gamma \nu$ is a Gibbs measure for any $\gamma \in \Gamma$.

Let $\nu$ be a $\Gamma$-invariant measure on $A^{V(\cover)}$.
Let $B$ be another finite alphabet. A measurable map $\tau$ from $A^{V(\cover)}$ to $B^{V(\cover)}$ is  $\Gamma$-equivariant if the following holds:
\begin{equation*}
\tau(\gamma \omega) = \gamma (\tau (\omega)),
\end{equation*}
for any $\gamma \in \Gamma$ and $\omega \in A^{V(\cover)}$.

Let $\nu$ be the unique Gibbs measure for $(\cover, \psi)$, and $\tau$ be a $\Gamma$-equivariant measurable map $A^{V(\cover)} \to B^{V(\cover)}$. Let $\mu$ be $\tau(\nu)$. We note that for a fixed edge $e \in E(G)$ and for any two edges $e',e'' \in E(\cover)$ such that $\phi(e') = \phi(e'') = e$, the marginal distributions $\pr_{e'}(\mu)$ and $\pr_{e''}(\mu)$ on $B \times B$ are the same(up to the natural re-labeling). So it makes sense to define the distribution $\mu_e$ on $B^{\lbrace u,v \rbrace} = B \times B$, where $e = \lbrace u,v \rbrace$.

Remind that the Shannon entropy of a measure $\eta$ on a finite set $Q$ is defined as 
\begin{equation*}
\entr(\eta) = -\sum_{q \in Q} \eta(\lbrace q\rbrace)\log \eta(\lbrace q\rbrace), 
\end{equation*}
with the convention $0 \log 0 = 0$. The Shannon entropy of a random variable is defined as the Shannon entropy of its distribution.

The main theorem of this paper is the following:
\begin{theor}\label{th: general edge-vertex}
In the notation above, the following inequality holds:

\begin{equation*}
\sum_{v \in V(G)}(\deg(v) - 1) \entr(\mu_v) \leq \sum_{e \in E(G)}\entr(\mu_e).
\end{equation*}

\end{theor}

The proof will ensue in the next section as a consequence of two lemmata concerning colorings of finite random coverings. Basically, the first one shows that the difference between the right-hand and the left-hand sides of the edge-vertex inequality is involved in the estimate of the expected number of some ``good'' colorings of random covers of graph $G$. On the other hand, the second one shows that this expected number is at least $1-o(1)$ (a ``good'' coloring exist almost surely).

Note that Theorem \ref{th: general edge-vertex} is exactly the ``general edge-vertex inequality'' from \cite{BGH}, the only change is that we consider factors of uniqie Gibbs measure processes instead of factors of IID's.  
IID processes form a particular case of unique Gibbs measure processes: it is easy to construct a potential whose unique Gibbs measure would be any given product-measure (the ``interactionless potential'').
Theorem \ref{th: decay of mutual info} and other entropy inequalities follow by constructing specially tailored factors, see \cite[Section 3]{BGH} for details. These constructions work in our case without any noticeable changes. 

A {\em $\Gamma$-cellular map} is a $\Gamma$-equivariant {\em continuous} map $\tau$ from $A^{V(\cover)}$ to $B^{V(\cover)}$. The continuity assumption may sound a bit fancy. In fact, it is equivalent to say that for any $v \in V(\cover)$ there is a finite set $W$ of $V(\cover)$ such that $(\tau(\omega))(v)$ depends only on $\pr_W(\omega)$. This equivalence is known as the Curtis-Hedlund-Lyndon theorem. 
There is a well-known trick to approximate a measurable equivariant map by a cellular map. Namely, if $\tau'$ is a measurable equivariant map, then  for any $\varepsilon>0$ there is a cellular map $\tau$ such that for any $v \in V(\cover)$ holds 
\begin{equation*}
\nu\Big(\big\{ \omega \in A^{V(\cover)} :  (\tau(\omega))(v) = (\tau'(\omega))(v)  \big\}\Big) > 1 - \varepsilon.
\end{equation*}
So it is enough to consider only factors obtained by cellular maps to prove Theorem \ref{th: general edge-vertex}.

\section{Random coverings and proof of Theorem \ref{th: general edge-vertex}} 

Let $G$ be a finite simple connected graph. Let $\cover$ be its universal covering together with a natural covering map $\varphi: \cover \to G$.
An $N$-fold covering of $G$ is a graph $\WGN$ whose vertex set is $\lbrace 1,\ldots,N \rbrace \times V(G)$ together with a graph morphism  $q_N: \WGN \to G$ such that 
\begin{enumerate}
\item for each $v \in V(G)$ and $i$ from $\lbrace 1, \ldots, N\rbrace$ holds $q_N(i, v) = v$;
\item there are exactly $N$ preimages of each edge from $G$;
\item  for each edge $\lbrace u,v\rbrace$, the preimage $q_N^{-1}(\lbrace u,v \rbrace)$ is a perfect matching between the sets $q_N^{-1}(u)$ and $q_N^{-1}(v)$;
\item if vertices $u'$ and $v'$  from $\WGN$ are connected, then $q_N(u')$ and $q_N(v')$ are connected.
\end{enumerate}
We note that, for each connected component of graph $\WGN$, there is a covering map $q: \cover \to \WGN$ onto this component such that $\varphi = q_N \circ q $

A uniform random $N$-fold covering of $G$ is an $N$-fold covering such that the perfect matchings $q_N^{-1}(e)$ for each $e \in E(G)$ are taken uniformly at random.
An important property of the random $N$-fold covering is that for big $N$ it looks like the universal covering in a neighborhood of almost every of its vertices. 
To be more precise, consider an $N$-fold covering $\WGN$. Take a vertex $u$ from $V(\WGN)$. Let $q: \cover \to \WGN$ be a covering map to the connected component of $\WGN$ containing $u$. We will say that vertex $u$ is $R$-nice, for a positive integer $R$, if $q$ is a bijection from $\Ball(u',R)$ to $\Ball(u,R)$, where $u'$ is any preimage of $u$ under map $q$. 
We note that the latter definition is independent of the choice of covering map $q$. We will say that an edge is $R$-nice if both of its end are $R$-nice. For each vertex $v$ in graph $G$, denote $L_v^{\WGN}$ the preimage set of this vertex under the natural covering map to $G$. We establish similar notation $L_e^{\WGN}$ for an edge $e \in E(G)$.

For a positive $\varepsilon$ and positive integer $R$, we will say that an $N$-fold covering $\WGN$ of graph $G$ is $(R,\varepsilon)$-nice if for each vertex $v \in V(G)$  the portion of $R$-nice vertices in $L_v^{\WGN}$ is bigger than $1-\varepsilon$, and the same holds for each edge.
The following lemma is discussed in \cite[Section 5.a]{BGH}

\begin{lem}\label{lem: random covering is nice}
Let $R$ be a positive integer and let $\varepsilon>0$. Then for any $\varepsilon>0$ there is $N'$ such that for every $N>N'$ the random uniform $N$-fold covering $\WGN$ is $(R,\varepsilon)$-nice with probability at least $1-\varepsilon$.
\end{lem}

\begin{proof}[Sketch of proof]
Let $v$ be a vertex from graph $G$ and $p$ be a finite non-backtracking path on $G$ starting at $v$. Let $u \in V(\WGN)$ be any vertex with $v = q_N(u)$, and let $\hat{p}$ be the lift of path $p$ to the path on $\WGN$ starting at $u$, and whose projection under $q_N$ is $p$. It is sufficient to prove that, with probability close to $1$, $\hat{p}$ has no recurring vertices provided that $N$ is big enough. This can be done easily by induction and using the ``lazy evaluation'' trick. Namely, instead of constructing at once the random perfect matchings involved in the definition of random covering, one should only pick randomly the edge needed to continue lifting the path considered. 
\end{proof}

Let $\WGN$ be an $N$-fold covering of $G$, let $B$ be a finite alphabet. Let $c$ be a map from $V(\WGN)$ to $B$. We associate with it the ``empirical distribution''. Namely, for each vertex $v$ we pick uniformly its preimage $v'$ under the covering map $q_N$. The distribution of $c(v')$, a measure on $B$, is our empirical distribution and will be denoted $\mu_v^c$. In other words, 
\begin{equation*}
\mu_v^c = \frac{1}{N}\sum_{v' \in L_v^{\WGN}} \delta_{c(v')}.
\end{equation*}
In the same way, for each edge $e = \lbrace u, v\rbrace$, we pick uniformly at random its preimage $e' = \lbrace u',v'\rbrace$ ($q_N(u')= u$ and $q_N(v')= v$) under the covering map. The distribution of $(c(u'),c(v'))$ is a measure on $B^e = B \times B$ and will be denoted $\mu_e^c$. We may also write:
\begin{equation*}
\mu_e^c = \frac{1}{N}\sum_{\lbrace u, v\rbrace \in L_e^{\WGN}} \delta_{(c(u'), c(v'))}.
\end{equation*}


In the next lemma we assume for a moment that $(\mu_v)_{v \in V(G)}$ and $(\mu_e)_{e \in E(G)}$ are just two collections of measures on $B$ and $B \times B$ respectively. We will say that these collections are consistent if $\pr_{v}(\mu_e) = \mu_v$ for every $e \in E(G)$ and $v$ is one of two vertices connected by edge $e$.
\begin{lem}
Let $(\mu_v)_{v \in V(G)}$ and $(\mu_e)_{e \in E(G)}$ be a consistent collection of measures. For every $\varepsilon>0$ the following holds:
\begin{multline}
\expect_{\WGN}\Big| c: V(\WGN) \to B, \lVert \mu_e^c - \mu_e\rVert \leq \varepsilon \text{ for every } e \in E(G) \Big|  \\
=\exp \bigg( N \bigg( \sum_{e \in E(G)} \entr(\mu_e) - \sum_{v \in V(G)} (\deg(v) -  1)\entr(\mu_v) + o_{ \varepsilon,N}(1)\bigg)\bigg),
\end{multline}
where o-small tends to zero as $N$ tends to infinity and $\varepsilon$ tends to zero.
\end{lem}
We omit the proof and refer the reader to \cite[Lemma 5.6]{BGH}. Note that the statement therein is slightly different, the precise equality is required. Nonetheless, the careful inspection of the proof reveals that it works in our case as well. 

Theorem \ref{th: general edge-vertex} will follow from the previous lemma if only we can prove that the expectation mentioned in that lemma is bigger than $1 - o(1)$. This is exactly what we get in the next one. 
\begin{lem}\label{lem: existence of good colorings}
For every $\varepsilon>0$ there is $N'$ such that for every $N>N'$, a random $N$-fold covering $\WGN$ with probability more than $1 - \varepsilon$ has a coloring $c: V(\WGN) \to B$ such that $\lVert \mu^c_e - \mu_e  \rVert \leq \varepsilon$ for every $e \in E(G)$.
\end{lem}

We will prove Lemma \ref{lem: existence of good colorings} in Section \ref{sec: nice coverings good colorings} after laying below the necessary groundwork on Gibbs measures.

\section{Asymptotic independence for unique Gibbs measures}

Distant parts of processes governed by unique Gibbs measures are almost independent. It turns out, that the latter is true in some sense even if we warp the Gibbs sructure outside of some (big enough) ball in the graph. the purpose of this section is to explain these points.

We refer the reader to the monographs \cite{G11} and \cite{RS15} for the preliminaries on Gibbs measures.

Gibbs measures satisfy the so-called {\em Markov property} (see \cite[p. 157]{G11}). Namely, let $\nu$ be a Gibbs measure for the potential $\psi$.
For any finite subset $\Lambda$ of $V(H)$ and for any $\sB(\Lambda)$-measurable  $\eltwo$-function $f$, holds
\begin{equation*}
\expect(f \vert \sB(V(H) \setminus \Lambda) ) = \expect(f \vert \sB(\partial^{2D} \Lambda)).
\end{equation*} 

\begin{rem}\label{rem: continuity of conditioning}
Consider the following expression:
\begin{equation*}\label{eq: conditional expectation approximation}
\lVert \expect(f \vert \sB(\partial^{2D} \Lambda)) - \expect(f)\rVert_2.
\end{equation*}
There is an explicit parameter: the measure $\nu$. In fact, only the projection $\nu_0= \pr_{\Lambda \cup \partial^{2D}\Lambda} \nu$ matters in this respect. We may notice also that $f$ is well-defined on $A^{\Lambda \cup \partial^{2D}\Lambda}$ and that $\nu'_0$ is positive on all elements of $A^{\Lambda \cup \partial^{2D}\Lambda}$. The latter means that the function
\begin{equation*}
\nu_1 \mapsto \lVert \expect_{\nu_1}(f \vert \sB(\partial^{2D} \Lambda)) - \expect(f)\rVert_2
\end{equation*}
is continuous in a neighborhood of $\nu_0$.
\end{rem}

Let $H$ be a locally-finite connected graph with a $D$-potential $\psi$. Let $H'$ be another locally-finite graph with a $D$-potential $\psi'$. Let $v$ be a vertex in $H$, $v'$ be a vertex in $H'$ and $R>D$ be an integer. A $(v,v',R)$-partial isomorphism is an isomorphism between induced subgraphs $t: \Ball(v,R) \to \Ball(v',R)$  such that $t(v)=v'$, and 
\[
\psi'_{t(u)}(t \omega) = \psi_u(\omega),
\] 
for every $u \in \Ball(v,R-D)$ and $\omega \in A^{\Ball(v,R)}$(note that $t$ induces a map $A^{\Ball(v,R)} \to A^{\Ball(v',R)}$). Roughly speaking, this means that the Gibbs structures are the same in the neighborhoods of $v$ and $v'$.

Suppose that the Gibbs measure is unique for potential $\psi$ on $H$ and $v$ is a vertex from $V(H)$. Let $f: A^{V(H)} \to \R$ be a bounded measurable function. Then 
\begin{equation}\label{eq: martingale convergence}
\lim_{R \to \infty }\expect\Big({f \vert \sB\big(V(H) \setminus \Ball(v,  R)\big)}\Big) = \expect (f),
\end{equation} 
where the limit is taken in the $\eltwo$-norm. 
The latter is a consequence of the martingale convergence theorem together with the fact that the tail subalgebra 
$\mathcal{T} = \bigcap_{R>0}\sB(V(H) \setminus \Ball(v,  R))$ is trivial on the unique Gibbs measure (\cite[Theorem 7.7 (a), p. 118]{G11}).

A unique Gibbs measure is stable in a certain sense. To be more precise, 
\begin{lem}\label{lem: stability of unique}
Let $\nu$ be the unique Gibbs measure for Gibbs structure $(H,\psi)$, let $\Lambda$ be a finite subset of $V(H)$ and $v \in \Lambda$. Then for every $\varepsilon>0$ there is such an integer $R$ that if the Gibbs structure is changed only outside of $\Ball(v,R)$, then any Gibbs measure $\nu'$ (it is not necessarily unique anymore) for the new Gibbs structure satisfies 
\begin{equation*}
\lVert \pr_{\Lambda}(\nu) - \pr_{\Lambda}(\nu') \rVert \leq \varepsilon.
\end{equation*} 
\end{lem}
\begin{proof}
Assume the contrary. There is an $\varepsilon>0$ and a sequence $(H^R, \psi^R)$ of Gibbs structures such that $(H^R, \psi^R)$ is isomorphic to $(H, \psi)$ in the $\Ball(v,R)$, and for any $(H^R,\psi^R)$ there is a Gibbs measure $\nu^R$ such that
\begin{equation*}
\lVert \pr_{\Lambda}(\nu) - \pr_{\Lambda}(\nu^R) \rVert > \varepsilon
\end{equation*}
if $\Lambda \subset \Ball(v,R)$.
For each $W \Subset V(H)$, the projection $\pr_W(\nu^R)$ is well-defined as soon as $W \subset \Ball(v,R)$.
Using the diagonal trick, we may extract a subsequence, and assume that projections $\pr_W (\nu^i)$ converge for all finite subsets $W$ of $V(H)$. For each $W$ consider the limiting measure $\nu_W$. Using the Kolmogorov theorem, we may extend this collection of measures to a measure $\nu_1$ on $A^{V(H)}$. This has a property that $\lVert \pr_{\Lambda}(\nu_1) - \pr_{\Lambda}(\nu) \rVert > \varepsilon$.
We also note that $\nu_1$ is a Gibbs measure for $(H, \psi)$. This contradicts our assumption that $\nu$ is the unique Gibbs measure for $(H, \psi)$. 
\end{proof}

\begin{lem}\label{lem: covariance bound}
Suppose that graph $H$ with potential $\psi$ has a unique Gibbs measure.
Let $f_1$ be a function defined on $A^{W_1}$ for some finite subset ${W_1}$ of $V(H)$, such that $0 \leq f_1 \leq 1$, and let $v \in W_1$. For every $\varepsilon>0$ there are such positive integers $R' < R''$ with $W_1 \subset (\Ball(v',R'))$ that if $(H,\psi)$ is $(v,v',R'')$-partially isomorphic to $(H',\psi')$, and $\nu'$ is a Gibbs measure (that is not necessarily unique) for $(H',\psi')$, then the following holds. 

For any function $f_2$ defined on $A^{W_2}$, where $W_2 \subset V(H') \setminus \Ball(v', R')$, such that $0 \leq f_2 \leq 1$, we have: 
\begin{equation*}
\big\lvert\Cov_{\nu'}(f_1',f_2) \big\rvert\leq \varepsilon, 
\end{equation*} 
where $f_1'$ is the transfer of $f_1$ from $H$ to $H'$ via the partial isomorphism map.
\end{lem}
\begin{proof}
Denote $\sA = \sB(V(H') \setminus \Ball(v',R'))$ and $\sC = \sB(\partial^{2D}\Ball(v',R'))$.
\begin{multline*}
\big\lvert \Cov(f'_1,f_2)\big\rvert  = \big\lvert \expect(f'_1 - \expect f'_1)(f_2 - \expect f_2)\big\rvert =\\
\Big\lvert\expect\Big(\expect\big((f'_1 - \expect f'_1)(f_2 - \expect f_2) \,\big\vert \sA\big)\Big)\Big\rvert = \\ 
\Big\lvert\expect\Big(\expect\big((f'_1 - \expect f'_1) \,\big\vert \sA\big) \cdot \big(f_2 - \expect f_2 \big)\Big)\Big\rvert.
\end{multline*}
The last equality is due to the fact that if $F_1$ and $F_2$ are two $\eltwo$-functions and $F_2$ is $\sA$-measurable, then $\expect\big( F_1 \cdot F_2 \vert \sA \big) = \expect\big( F_1 \vert \sA \big) \cdot F_2$ almost everywhere. We carry on using the Cauchy-Bunyakovsky-Schwartz inequality and then the Markov property:
\begin{multline*}
\Big\lvert\expect\Big(\expect\big((f'_1 - \expect f'_1) \,\big\vert \sA\big) \cdot \big(f_2 - \expect f_2 \big)\Big)\Big\rvert \leq \Big\lVert\expect\big((f'_1 - \expect f'_1) \,\big\vert \sA\big)\Big\rVert_2 =\\ \big\lVert \expect(f'_1 \;\vert \sA ) - \expect f'_1 \big\rVert_2 = \big\lVert \expect(f'_1 \;\vert \sC ) - \expect f'_1 \big\rVert_2.
\end{multline*}
If $R''$ is big enough, then we can transfer the expression from graph $H'$ back to $H$ using Remark \ref{rem: continuity of conditioning} and Lemma \ref{lem: stability of unique}. Then we apply again the Markov property to change the conditioning subalgebra: 
\begin{equation*}
\big\lVert \expect(f'_1 \;\vert \sC ) - \expect f'_1 \big\rVert_2 \leq \varepsilon/2 + \big\lVert \expect(f_1 \;\vert \sC' ) - \expect f_1 \big\rVert_2 = \varepsilon/2 + \big\lVert \expect(f_1 \;\vert \sA' ) - \expect f_1 \big\rVert_2,
\end{equation*}
where $\sC' = \sB(\partial^{2D}\Ball(v,R'))$ and  $\sA' = \sB(V(H') \setminus \Ball(v',R'))$. Now we use the martingale convergence (equation \ref{eq: martingale convergence}) to note that 
\begin{equation*}
\big\lVert \expect(f_1 \;\vert \sA' ) - \expect f_1 \big\rVert_2 \leq \varepsilon/2,
\end{equation*}
for big enough $R'$. Alltogether, if we take $R'$ and $R''$ big enough and such that $R'< R''$, then we have 
\begin{equation*}
\big\lvert\Cov_{\nu'}(f_1',f_2) \big\rvert\leq \varepsilon.
\end{equation*}
\end{proof}
 

\section{Nice coverings have good colorings}\label{sec: nice coverings good colorings}

Suppose $\WGN$ is an $N$-fold covering of graph $G$. We note that in this case there is a natutal translation of the potential to this covering. Indeed, for any $u \in V(\WGN)$ let $q$ be a covering map from $H$ to the connected component of $\WGN$ containing $u$. Fix any preimage $u'$ of $u$ under map $q$. Any coloring $\omega \in A^{V(\WGN)}$ can be lifted to a periodic coloring $\omega' \in A^{V(\cover)}$. We now define $\psi^{\WGN}_u(\omega)$ to be $\psi_{q^{-1}(u)}(\omega')$. This construction looks very natural when $D$-ball around $u$ is isomorphic to the corresponding $D$-ball around $u'$. In the latter case we literally translate the potential from $\Ball(u',D)$ to $\Ball(u,D)$. 
In a similar fashion, we can transfer the cellular map $\tau$ to the map $\tau^{\WGN}$ from $A^{V(\WGN)}$ to  $B^{V(\WGN)}$. Namely, we will define $(\tau^{\WGN}(\omega))(u)$ to be $(\tau(\omega'))(u')$.
The transferred potential will define the Gibbs measure  $\nu^{\WGN}$ on $A^{\WGN}$(note that it is defined on a finite set and is unique therefore). Now define $\mu^{\WGN} = \tau^{\WGN}(\nu^{\WGN})$. Remind that for a map $c: \WGN \to B$ and any $e \in E(G)$ we defined the empirical ditribution $\mu_e^c$.

\begin{lem}\label{lem: nice covering has nice coloring}
For any $\varepsilon > 0 $ there is $\delta>0$ and natural numbers $R$ and $N'$, that if an $N$-fold covering $\WGN$ ($N>N'$) of graph $G$ is $(R,\delta)$-nice, then there is a coloring $c: V(\WGN) \to B$ such that $\lVert \mu_e - \mu_e^c \rVert \leq \varepsilon$ for every $e \in E(G)$.
\end{lem}
 
Note that this lemma implies immediately Lemma \ref{lem: existence of good colorings} since, by Lemma \ref{lem: random covering is nice}, a random covering $\WGN$ for $N>N'$ is $(R,\delta)$-nice with probability bigger than $1-\varepsilon$, provided that $N'$ is big enough. 

\begin{proof}

As in the proof of Lemma 5.4 from \cite{BGH} we will employ a random construction. We will prove that for big enough $N'$ and $R$, and small enough $\delta>0$, a $\mu^{\WGN}$-random $c$ will suffice for us with high probability.

Fix an edge $e \in E(G)$ connecting two vertices $u$ and $v$ of graph $G$.

We note that $\expect_c(\mu_e^c)$ is arbitrarily close to $\mu_e$, provided $N'$, $R$ are large and $\delta$ is small enough. This is due to two observations. First, if $R$ is big, and $i$ is an $R$-nice edge of $\WGN$, then the distribution of $B \times B$-valued random varible $(c(u'),c(v'))$ (where $u'$ and $v'$ are the vertices connected by edge $i$, and these are the respective preimages of vertices $u$ and $v$) is close to $\mu_e$. Second, the contribution of the non-nice vertices is controlled by $\delta$ and could be made arbirarily small.

Let us show now that actually $\mu^c_e$ concentrates near its expected values. 
Take any two elements $b_u, b_v$ from $B$. Edge $e$ has a preimage set $L_e^{\WGN}$ of edges from $\WGN$. Define $(f_{i,e,b_u,b_v})_{i \in L_e^{WGN}}$ to be the variable that equals $1$ whenever $c(u')=b_u$ and $c(v')=b_v$, and $0$ otherwise. Denote 
\begin{equation*}
\overline{f_{e,b_u,b_v}} = \frac1{N} \sum_{i \in L_e^{\WGN}} f_{i,e,b_u,b_v}.
\end{equation*}
Note that $\overline{f_{e,b_u,b_v}} = \mu^c_e(\lbrace (b_u,b_v)\rbrace)$. 
Here lies the main difference of our proof from that of \cite{BGH}. In that paper, a measure concentration bound was used. Since our process is not IID, we need to find another option.
By the Chebyshev inequality, it amounts to prove the following:

\begin{lem}\label{lem: mean value}
For any $\varepsilon>0$ there are large enough $N'$ and $R$, and small enough $\delta>0$, such that the following holds. For any $(R,\delta)$-nice $N$-fold covering $\WGN$ with $N>N'$, we have:
\begin{equation*}
\Var( \overline{f_{e,b_u,b_v}}) < \varepsilon.
\end{equation*}

\end{lem}
The rest is devoted to the proof of the lemma above.

Let's explain the variant of the second-moment argument needed. Let $(Y_i)_{i =1 \ldots N}$ be a collection of random variables on the same probability space with zero expectation each (we can alway obtain the latter by subtracting respective expectation). Assume also that $\lvert Y_i \rvert  \leq 1$ almost surely. Let
\begin{equation*}
\overline{Y} = \frac{1}{N}\sum_{1 \leq i \leq N} Y_i.
\end{equation*}

We want to show that for big $N$, the second moment of $\overline{Y}$ is close to $0$. Of course we will need some additional assumptions.
Let's expand:

\begin{equation}\label{eq: second moment}
\expect \overline{Y}^2 = \frac{1}{N^2} \sum_{ i } \expect Y_i^2 + \frac{1}{N^2}\sum_{i \neq j} \expect  Y_i  Y_j.
\end{equation}
Note that the first summand vanishes as $N$ gets bigger.

If $Y_i$'s are pairwise independent, then the second term of expression \ref{eq: second moment} is zero, and we are done. This is classical Chebyshev's law of large numbers. Let us weaken the independence assumption a little. Assume that each $Y_i$ ``interacts'' with the uniformly bounded number of $Y_j$'s. That is there is a contant $K$(uniform in $N$) such that for each $i$ there are at most $K$ of $j$'s with $Y_i$ and $Y_j$ being pairwise dependent. In that case the second summand of \ref{eq: second moment} is bounded by $K/N$, and we obtain the desired.

For our exposition we need a further weakening. Roughly speaking, we say that that each $Y_i$ interacts in a ``significant way'' only to a uniformly bounded number of $Y_j$'s. We assume that all pairs can interact, but for each $\varepsilon'>0$ there is a number $K_{\varepsilon'}$(uniform in $N$) such that for each $i$ there are at most $K_{\varepsilon'}$ such $j$'s that $\lvert \expect Y_i Y_j\rvert > \varepsilon'$. 
So there are ``far'' pairs $(i,j)$ with covariance less than $\varepsilon'$ and the ``close'' ones with covariance not less than $\varepsilon'$.
This assumption implies that the second term of the second moment is bounded by $\varepsilon' + K_{\varepsilon'}/N$. 

We need to take into account now that some elements $i$ are ``bad'' = ``not nice''(we don't have the bound on interaction for them as described above), but their number is limited by $\delta N$. In our expression for the second moment we will get:

\begin{multline}\label{eq: enhanced second moment}
\expect \overline{Y}^2 = \frac{1}{N^2} \sum_{ i } \expect Y_i^2 + \frac{1}{N^2}\sum_{i \neq j} \expect  Y_i  Y_j =\\ \frac{1}{N^2} \sum_{ i } \expect Y_i^2 + \frac{1}{N^2}\sum_{i \neq j \atop \text{nice}} \expect  Y_i  Y_j +  \frac{1}{N^2}\sum_{i \neq j \atop \text{$i$ or $j$ bad}} \expect  Y_i  Y_j \leq \\ 
\frac{1}{N} + \varepsilon' + \frac{K_{\varepsilon'}}{N} + 2 \delta.
\end{multline}

We return to the proof of Lemma \ref{lem: mean value}. 
By Lemma \ref{lem: covariance bound}, we can take $R$ so big that for any two $R$-nice edges $i,j \in L_e^{\WGN}$ that are $2R$ apart(in the usual graph distance on $\WGN$), we will have 
\begin{equation*}
\Cov(f_{i,e,b_u,b_v},f_{j,e,b_u,b_v}) \leq \varepsilon/4.
\end{equation*}
So for any $R$-nice edge $i$ from $L_e^{\WGN}$ there are no more than  $2R (\deg G)^{2R+2}$ $R$-nice edges $j$  from $L_e^{\WGN}$ such that 

\begin{equation*}
\Cov(f_{i,e,b_u,b_v},f_{j,e,b_u,b_v}) > \varepsilon/4.
\end{equation*}
So in \ref{eq: enhanced second moment} we have $\varepsilon' = \varepsilon/2$ and $K_{\varepsilon'} = 2R (\deg G)^{2R+2}$. Now if we take $\delta$ smaller than $\varepsilon/2$, then for big enough $N$, the variance of $\overline{f_{e,b_u,b_v}}$ will be smaller than $\varepsilon$. 
This finishes the proof of Lemma \ref{lem: mean value} and thus of Lemmata  \ref{lem: nice covering has nice coloring} and \ref{lem: existence of good colorings}. 

\end{proof}


\begin{thebibliography}{100000}
\bibitem[A15]{A15} A. Alpeev, 
\textit{The entropy of Gibbs measures on sofic groups} Zap. Nauchn. Sem. POMI, 436 (2015): 34--48.

\bibitem[A17]{A17} A. Alpeev, 
\textit{A random ordering formula for Gibbs measures for the sofic and Rokhlin entropy of Gibbs measures}, arXiv preprint arXiv:1705.08559 (2017).

\bibitem[AuP17]{AuP17} T. Austin and M. Podder, \textit{Gibbs measures over locally tree-like graphs and percolative entropy over infinite regular trees}, arXiv preprint arXiv:1705.03589 (2017).

\bibitem[BS14]{BS14} A. Backhausz and B. Szegedy, \textit{On large girth regular graphs and random processes on trees}, arXiv preprint arXiv:1406.4420 (2014).

\bibitem[BGHV]{BGHV} A. Backhausz, B. Gerencs\'er, V. Harangi, and M. Vizer, \textit{Correlation bound for distant parts of factor of iid processes}, Combin. Probab. Comput., published online (2017).

\bibitem[BSV15]{BSV15} A. Backhausz, B. Szegedy, and B Vir\'ag, \textit{Ramanujan graphings and correlation decay in local algorithms}, Random Structures Algorithms, 47(3):424–435, (2015).

\bibitem[BGH]{BGH} A. Backhausz, B. Gerencs\'er and V. Harangi, \textit{Entropy inequalities for factors of IID}, to appear in Grops, Geometry and Dynamics, arXiv preprint arXiv:1706.04937 (2018).

\bibitem[Ba]{Ba} R. J. Baxter \textit{Exactly solved models in statistical mechanics}, Academic Press, 1982.

\bibitem[CHV]{CHV} E. Cs\'oka, V. Harangi, and B. Vir\'ag \textit{Entropy and expansion}, arXiv preprint arxiv:1811.09560 (2018).

\bibitem[DM10a]{DM10a} A. Dembo and A. Montanari, \textit {Gibbs measures and phase transitions on sparse random graphs}, Brazilian Journal of Probability and Statistics (2010): 137-211.

\bibitem[DM10b]{DM10b} A. Dembo and A. Montanari, \textit{Ising models on locally tree-like graphs}, The Annals of Applied Probability 20.2 (2010): 565-592.

\bibitem[EW11]{EW11} M. Einsiedler and T. Ward. \textit{Ergodic theory with a view towards number theory}. Graduate texts in mathematics, 259. Springer, London, 2011.

\bibitem[G11]{G11} H.-O. Georgii, \textit{Gibbs measures and phase transitions}. Vol. 9. Walter de Gruyter, 2011.

\bibitem[GH17]{GH17} B. Gerencs\'er and V. Harangi, \textit{Mutual information decay for factors of iid}, Ergodic Theory and Dynamical Systems, to appear, arXiv:1703.04387 (2017).






\bibitem[RS15]{RS15} F. Rassoul-Agha and T. Seppäläinen, \textit{A course on Large Deviations with an Introduction to Gibbs Measures}, Vol. 162. American Mathematical Soc., 2015.





%










\end{thebibliography}
\end{document}